\documentclass[12pt, reqno]{amsart}
\usepackage{amsmath, amsthm, amscd, amsfonts, amssymb, graphicx, color}
\usepackage[bookmarksnumbered, colorlinks, plainpages]{hyperref}

\makeatletter \oddsidemargin.9375in \evensidemargin \oddsidemargin
\newtheorem{theorem}{Theorem}[section]
\newtheorem{lemma}[theorem]{Lemma}

\newtheorem{corollary}[theorem]{Corollary}
\theoremstyle{definition}

\theoremstyle{remark}
\newtheorem{remark}[theorem]{Remark}
\numberwithin{equation}{section}

\begin{document}
\setcounter{page}{1}


\title[On refined Young inequality and its reverse ]{On refined operator version
of Young inequality and its reverse}
\author[A. Sheikhhosseini and M. Khosravi]{Alemeh
Sheikhhosseini$^1$  Maryam Khosravi$^{2*}$}

\address{$^{1}$ Department of
 Pure Mathematics, Faculty of Mathematics and Computer,
 Shahid Bahonar University of Kerman,
Kerman, Iran} \email{sheikhhosseini@uk.ac.ir;
hosseini8560@gmail.com}

\address{$^{2}$ Department of
 Pure Mathematics, Faculty of Mathematics and Computer,
 Shahid Bahonar University of Kerman,
Kerman, Iran} \email{khosravi$_-$m @uk.ac.ir;
khosravi$_-$m2000@yahoo.com}

\subjclass[2010]{Primary: 47A63; Secondary: 47A64, 15A42.}
\keywords{Young inequality, positive operators, weighted means,
Hilbert-Schmidt norm\\$*$ Corresponding author. }

\begin{abstract}
In this note, some refinements of Young inequality and its reverse
for positive numbers are proved and using these inequalities some
operator versions and Hilbert-Schmidt norm versions for matrices
of these inequalities are obtained.
\end{abstract}
\maketitle
\section{Introduction}
Let $\mathbb{B}(\mathcal{H})$ denote the $ C^{*}$-algebra of all
bounded linear operators on a complex Hilbert space $ \mathcal{H}.
$ In the case when dim $ \mathcal{H}=n,  $ we identify
$\mathbb{B}(\mathcal{H})$ with the matrix algebra  $
\mathbb{M}_{n} $ of all $ n \times n $ complex matrices.
  For  $ A=(a_{ij})  \in \mathbb{M}_{n}, $ the Hilbert-Schmidt norm of $ A $ is
  defined by $$ \|A\|_{2} = \left( \sum_{i, j=1}^{n}  |a_{ij}
  |^{2}\right)^{\frac{1}{2}}=\sum_{i=1}^n \left(s^2_j(A)\right)^{\frac{1}{2}},$$
  where $s_j(A)\ (1\leq j\leq n)$ are the singular values of $A$.
    It is  known that $ \|.\|_{_{2}} $  is a  unitarily invariant norm.

  Let $ A, B \in \mathbb{M}_{n}. $  Denoted by $ A \circ B $ the Schur (Hadamard)
  product of $ A $ and $ B, $ that is, the entrywise product.

  For positive real numbers $ a $ and $ b $, the classical Young inequality says
  that if $ \nu \in [0, 1], $ then
  $$ a^{1-\nu}b^{\nu} \leqslant (1- \nu) a +  \nu b, $$
  with equality if and only if $ a=b. $ When $ \nu=\frac{1}{2}, $ the Young
  inequality is called the arithmetic-geometric mean inequality
  \begin{equation}
  \sqrt{ab} \leqslant \dfrac{a+b}{2}.
  \end{equation}
  Throughout, we denote $ a^{1-\nu}b^{\nu} $ and $ (1-\nu ) a + \nu b, $ respectively
   by $ a\sharp_{\nu}  b$ and  $ a \nabla_{\nu} b. $
   The Heinz mean is defined as
   $$  H_{\nu}(a, b)= \frac{a^{1-\nu}b^{\nu} + a^{\nu}b^{1-\nu} }{2}$$
   for $ a, b \geqslant 0 $ and $ \nu \in [0, 1]. $  It's easy to see that
   $$ \sqrt{ab} \leqslant  H_{\nu}(a, b)  \leqslant  \dfrac{a+b}{2}.$$
  In \cite{k1} and \cite{k2}, F. Kittaneh and Y. Manasrah  improved   the Young
  inequality and its reverse as follows:
  \small{
  \begin{equation}\label{re1}
  a^{1-\nu}b^{\nu} + r(\sqrt{a}-\sqrt{b})^{2} \leqslant (1-\nu) a + \nu b \leqslant
  a^{1-\nu}b^{\nu} + s(\sqrt{a}-\sqrt{b})^{2},
  \end{equation}}
  where $ r=\min \{\nu, 1-\nu\} $  and $ s=\max \{\nu, 1-\nu \}. $

  The authors of  \cite{he}  and \cite{oh}  obtained another  refinement of the  Young inequality as follows:
   \small{
    \begin{equation}\label{re2}
   r^{2}(a-b)^{2}    \leqslant    (  (1-\nu) a + \nu b )^{2}-(a^{1-\nu}b^{\nu})^{2}
   \leqslant s^{2}(a-b)^{2},
    \end{equation}}
  where  $ r=\min  \{\nu, 1-\nu   \}$, and $ s=\max \{\nu, 1-\nu \}. $

  Recently, J. Zhao and J. Wu \cite{jz} obtained the following  refinement of  inequality
  (\ref{re1}): \small{
  \begin{align*}
  r ((ab)^{\frac{1}{4}}-\sqrt{a})^{2} + \nu (\sqrt{a}-\sqrt{b})^{2}
  &\leqslant (1-\nu) a+ \nu b-a^{1-\nu}b^{\nu }\\
&  \leqslant (1- \nu) (\sqrt{a}-\sqrt{b})^{2} -r
((ab)^{\frac{1}{4}}-\sqrt{b})^{2},
  \end{align*}}
  where $ 0 \leqslant \nu \leqslant \frac{1}{2} $ and  $ r=\min \{ 2\nu, 1- 2\nu \} $ and
  \small{
  \begin{align*}
  r ((ab)^{\frac{1}{4}}-\sqrt{b})^{2} +(1- \nu) (\sqrt{a}-\sqrt{b})^{2}
 &\leqslant (1-\nu) a+ \nu b- a^{1-\nu}b^{\nu }\\
 & \leqslant  \nu
(\sqrt{a}-\sqrt{b})^{2} -r ((ab)^{\frac{1}{4}}-\sqrt{a})^{2},
  \end{align*}}
  where $ \frac{1}{2} \leqslant \nu \leqslant 1 $ and  $ r=\min \{2(1-\nu), 1- 2(1-\nu)  \}. $\\
  Also, they obtained the following  refinement of  inequalities
  (\ref{re2}):
  \small{
  \begin{align}\label{e10}
  r (\sqrt{ab}-a)^{2} + \nu^{2} (a-b)^{2}&  \leqslant ((1-\nu) a+
  \nu b)^{2}-(a^{1-\nu}b^{\nu })^{2} \nonumber \\
 &  \leqslant (1- \nu)^{2} (a-b)^{2} -r (\sqrt{ab}-b)^{2},
  \end{align}
  }
  where $ 0 \leqslant \nu \leqslant \frac{1}{2} $ and  $ r=\min \{ 2\nu, 1- 2\nu \} $ and
  \small{
  \begin{align}\label{e11}
  r (\sqrt{ab}-b)^{2} +(1- \nu)^{2} (a-b)^{2}&   \leqslant ((1-\nu) a+
  \nu b)^{2}-(a^{1-\nu}b^{\nu })^{2} \nonumber \\
 &  \leqslant \nu^{2} (a-b)^{2} -r (\sqrt{ab}-a)^{2},
  \end{align}}
  where $ \frac{1}{2} \leqslant \nu \leqslant 1 $ and  $ r=\min \{2(1-\nu), 1- 2(1-\nu)  \}
  $.

 Let  $ A, B \in  \mathbb{B}(\mathcal{H}) $  be two operators and $ \nu \in [0, 1].
 $ The
$ \nu- $weighted arithmetic mean of $ A $ and $ B  $,  denoted and
defined by:
$$ A \nabla_{\nu} B=(1-\nu) A+ \nu B.$$
If $ A $  is invertible, $ \nu- $geometric mean and   $ \nu- $Heinz mean  of $ A $ and $ B $ are  defined
 respectively,   as
$$A\sharp_{\nu}B=A^{\frac{1}{2}} (A^{-\frac{1}{2}}BA^{-\frac{1}{2}})^{\nu}A^{\frac{1}{2}} $$
and
$$H_{\nu} (A, B)=\dfrac{A\sharp_{\nu}B+ A\sharp_{1-\nu}B}{2}. $$
 In addition, if both $ A $ and $ B $ are invertible, $ \nu- $harmonic
 mean of $ A $ and $ B, $ denoted by $ A !_{\nu} B,  $ is defined as
    $$ A !_{\nu} B =((1-\nu)A^{-1}+\nu B^{-1})^{-1}.  $$
    When $ \nu=\frac{1}{2}, $ we write  $ A \nabla B, A\sharp B $ and  $ A !  B $
    for brevity, respectively.\\
It is well known that if $ A $ and $ B $ are positive invertible operators, then
\begin{equation*}\label{0e}
A \nabla_{\nu} B \geqslant A\sharp_{\nu}B \geqslant A!_{\nu}B,
\end{equation*}
for $  0 < \nu < 1; $  see \cite{ fu2,  fu1}  for more
information.

Based on the refined Young inequality (\ref{e10}) and its reverse
(\ref{e11}),  J. Zhao and J. Wu \cite{jz}  proved that if  $A, B,
X  \in  \mathbb{M}_{n} $  such that $ A $ and $ B $ are positive
semidefinite, then \small{
\begin{align}\label{b1}
\nu^{2} \|AX-XB\|_{2}^{2} +&r \|A^{ \frac{1}{2}}XB^{\frac{1}{2}}
 -AX\|_{2}^{2}\nonumber\\
 &\leqslant \| (1-\nu)AX+ \nu XB \|_{2}^{2}-\|A^{1-\nu}X B^{\nu}\|_{2}^{2} \nonumber \\
&  \leqslant (1- \nu)^{2} \|AX-XB\|_{2}^{2}+ -r \|A^{\frac{1}{2}}
XB^{\frac{1}{2}} -XB\|_{2}^{2},
\end{align}}
where $ 0 \leqslant \nu \leqslant \frac{1}{2} $ and  $ r=\min \{ 2\nu, 1- 2\nu \} $  and
\small{
\begin{align}\label{b2}
(1- \nu)^{2} \|AX-XB\|_{2}^{2}+&r \|A^{\frac{1}{2}}XB^{
\frac{1}{2}} -XB\|_{2}^{2}\nonumber\\
& \leqslant \| (1-\nu)AX+ \nu XB \|_{2}^{2}- \|A^{1-\nu}X B^{\nu}\|_{2}^{2}\nonumber \\
& \leqslant \nu^{2} \|AX-XB\|_{2}^{2}-r \|A^{\frac{1}{2}}X
B^{\frac{1}{2}} - A X\|_{2}^{2},
 \end{align}}
where $ \frac{1}{2}  \leqslant \nu  \leqslant 1  $ and  $ r=\min
\{2(1-\nu), 1- 2(1-\nu)  \}. $

Their results were generalized by Liao and Wu \cite{jmi}, using
Kantorovich constant.

Furthermore,  some similar results can be found in \cite{alzer,
feng7}.

In addition, in \cite{mo}, the authors investigated on these
inequalities, for the cases that $\nu\leq0$ or $\nu\geq1$. In
these cases, they proved the reverse of some of these
inequalities. Furthermore, in \cite{salemi}, the numerical version
of some of these relations, are discussed.

The main aim of this paper, is to state a generalization of these
inequalities. First, we present some generalizations of numerical
inequalities and base of them we prove some refined  operator
versions of Young inequality and its reverse. Also some
inequalities for Hilbert-Schmidt norm of matrices are obtained.

In this paper, for $0<\nu<1$, the notations $m_k=  \lfloor 2^k\nu
\rfloor $ is  the largest integer not greater than
 $2^k\nu$,  $ r_0=\min \{ \nu,
1-\nu\} $ and $ r_{k}=\min \{ 2r_{k-1}, 1-2r_{k-1} \}  $, for
$k\geq1$.
\section{Numerical results}
We start with some numerical results.
\begin{theorem}\label{l01}
Let $ a, b $ be two positive real numbers and $ \nu \in (0, 1). $
Then
\begin{equation}\label{y1}
 a\nabla_{\nu} b \geqslant a\sharp_{\nu}b+
\sum_{k=0}^{\infty}r_{k}\big[\big(a^{1-\frac{m_k}{2^k}}b^{\frac{m_k}{2^k}}
\big)^{\frac{1}{2}}-\big(a^{1-\frac{m_k+1}{2^k}}b^{\frac{m_k+1}{2^k}}
\big)^{\frac{1}{2}}\big]^{2}.
\end{equation}

In addition, if $\nu=\frac{t}{2^n}$ for some $t, n\in\mathbb{N}$,
then
\begin{equation*}
 a\nabla_{\nu} b = a\sharp_{\nu}b+
\sum_{k=0}^{n-1}r_{k}\big[\big(a^{1-\frac{m_k}{2^k}}b^{\frac{m_k}{2^k}}
\big)^{\frac{1}{2}}-\big(a^{1-\frac{m_k+1}{2^k}}b^{\frac{m_k+1}{2^k}}
\big)^{\frac{1}{2}}\big]^{2}.
\end{equation*}
\end{theorem}

\begin{proof}
 It is
enough to prove that for each $n\in\mathbb{N}\cup\{0\}$,
\begin{equation}\label{y02}
 a\nabla_{\nu} b \geqslant a\sharp_{\nu}b+
\sum_{k=0}^{n}r_{k}\big[\big(a^{1-\frac{m_k}{2^k}}b^{\frac{m_k}{2^k}}
\big)^{\frac{1}{2}}-\big(a^{1-\frac{m_k+1}{2^k}}b^{\frac{m_k+1}{2^k}}
\big)^{\frac{1}{2}}\big]^{2}.
\end{equation}
We prove it by induction. For $n=0$, we get to the well-known
inequality \eqref{re1}. Let inequality \eqref{y02} holds for $n$.

First, let $ 0 < \nu < \frac{1}{2}$. Thus,  we have
\begin{align*}
 a\nabla_{\nu} b -r_0(\sqrt{a}-\sqrt{b})^2&
 = a\nabla_{\nu} b -\nu(\sqrt{a}-\sqrt{b})^2\\
&=2\nu \sqrt{ab}+(1-2\nu)a\\
&=a\nabla_{2\nu}\sqrt{ab}
\end{align*}
Applying inequality \eqref{y02} for two positive numbers $a$ and
$\sqrt{ab}$ and $2\nu\in(0,1)$, we have
\begin{align*}
 a\nabla_{\nu} b
 -r_0(\sqrt{a}-\sqrt{b})^2&=a\nabla_{2\nu}\sqrt{ab}\\
 &\geq a\sharp_{2\nu}\sqrt{ab}+
 \sum_{k=0}^{n}r_{k+1}\big[\big(a^{1-\frac{m_{k+1}}{2^k}}(\sqrt{ab})^{
 \frac{m_{k+1}}{2^k}}
\big)^{\frac{1}{2}} \\&\quad
-\big(a^{1-\frac{m_{k+1}+1}{2^k}}(\sqrt{ab})^{\frac{m_{k+1}+1}{2^k}}
\big)^{\frac{1}{2}}\big]^{2}\\
&=a\sharp_{\nu}b+
\sum_{k=1}^{n+1}r_{k}\big[\big(a^{1-\frac{m_k}{2^k}}b^{\frac{m_k}{2^k}}
\big)^{\frac{1}{2}}-\big(a^{1-\frac{m_k+1}{2^k}}b^{\frac{m_k+1}{2^k}}
\big)^{\frac{1}{2}}\big]^{2}.
\end{align*}

For $\frac{1}{2}<\nu<1$, we can apply the first part for $1-\nu$
and replace $a$ and $b$. Note that $[2^k(1-\nu)]=2^k-[2^k\nu]-1$
if $2^k\nu$ is not integer. Thus, if $2^k\nu$ is not integer for
each $k$, the inequality follows.

Now, let $\nu=\frac{t}{2^{\ell}}$ for some odd number $t$ and
$\ell\in\mathbb{N}\cup\{0\}$. Since for each $i<\ell$, the
coefficient $r_i\leq\frac{1}{2}$ is of the form
$\frac{t_i}{2^{\ell-i}}$, it can be concluded that $r_{\ell}=0$
and so $r_k=0$ for all $k\geq n$. On the other hand $2^k\nu$ is
not integer for $k< \ell$. So the result follows.

A similar argument, shows the equality holds when
$\nu=\frac{t}{2^n}$.
\end{proof}
\begin{remark}
Note that the series appear in this theorem is a positive series
with a finite upper bound. So it is convergent. This fact is also
satisfies with all other series appear in this note.
\end{remark}

Changing the place of numbers $a$ and $b$ in inequality
\eqref{y1}, we can state the following result for Heinz mean.
\begin{corollary}
Let $ a, b $ be two positive real numbers and $ \nu \in (0, 1). $
Then
\begin{equation*}
 a\nabla b \geqslant H_{\nu}(a,b)+
\sum_{k=0}^{\infty}r_{k}\big[H_{\frac{m_k}{2^k}}(a,b)-2H_{\frac{2m_k+1}{2^{k+1}}}(a,b)
+H_{\frac{m_k+1}{2^k}}(a,b)\big].
\end{equation*}
\end{corollary}


In the  following theorem, we state a reverse of Young inequality.
\begin{theorem}\label{l2}
Let $ a, b $ be two positive real numbers and $ \nu \in (0, 1). $
Then
\begin{equation}\label{y2}
 a\nabla_{\nu} b \leqslant a\sharp_{\nu}b+(\sqrt{a}-\sqrt{b})^2-
\sum_{k=0}^{\infty}r_{k}\big[\big(a^{\frac{m_k}{2^k}}b^{1-\frac{m_k}{2^k}}
\big)^{\frac{1}{2}}-\big(a^{\frac{m_k+1}{2^k}}b^{1-\frac{m_k+1}{2^k}}
\big)^{\frac{1}{2}}\big]^{2}.
\end{equation}

\end{theorem}

\begin{proof}
By $a\sharp_{\nu}b+b\sharp_{\nu}a\geq2\sqrt{ab}$, and inequality
\eqref{y1}, we have
\begin{align*}
(\sqrt{a}-\sqrt{b})^2-a\nabla_{\nu}b&=b\nabla_{\nu}a-2\sqrt{ab}\\
&\geq
-a\sharp_{\nu}b+\sum_{k=0}^{\infty}r_{k}\big[\big(a^{\frac{m_k}{2^k}}b^{1-\frac{m_k}{2^k}}
\big)^{\frac{1}{2}}-\big(a^{\frac{m_k+1}{2^k}}b^{1-\frac{m_k+1}{2^k}}
\big)^{\frac{1}{2}}\big]^{2}.
\end{align*}
So the result follows.
\end{proof}

\begin{corollary}
Let $ a, b $ be two positive real numbers and $ \nu \in (0, 1). $
Then
\begin{equation*}
 a\nabla b \leqslant H_{\nu}(a,b)+(\sqrt{a}-\sqrt{b})^2-
\sum_{k=0}^{\infty}r_{k}\big[H_{\frac{m_k}{2^k}}(a,b)-2H_{\frac{2m_k+1}{2^{k+1}}}(a,b)
+H_{\frac{m_k+1}{2^k}}(a,b)\big].
\end{equation*}
\end{corollary}

\begin{remark}
Replacing $ a $ and $ b $ by their squares in (\ref{y1}) and (\ref{y2}), respectively, we obtain
\begin{equation}\label{y3}
  a^{2}\nabla_{\nu} b^{2} \geqslant a^{2}\sharp_{\nu}b^{2}+
 \sum_{k=0}^{\infty}r_{k} \big[  a^{1-\frac{m_k}{2^k}}b^{\frac{m_k}{2^k}}
  - a^{1-\frac{m_k+1}{2^k}}b^{\frac{m_k+1}{2^k}}
   \big]^{2}
\end{equation}
and

\begin{equation}\label{y4}
 a^{2}\nabla_{\nu} b^{2} \leqslant a^{2}\sharp_{\nu}b^{2}+(a - b)^2-
\sum_{k=0}^{\infty}r_{k}\big[a^{\frac{m_k}{2^k}}b^{1-\frac{m_k}{2^k}}
- a^{\frac{m_k+1}{2^k}}b^{1-\frac{m_k+1}{2^k}} \big]^{2}.
\end{equation}

\end{remark}
The following two theorems, are useful to prove a version of these
inequalities for the Hilbert-Schmidt norm of matrices.
\begin{theorem}\label{l3}
Let $ a, b $ be two positive real numbers and $ \nu \in (0, 1). $
Then
\begin{equation}\label{y5}
 ( a \nabla_{\nu} b)^{2} \geqslant ( a \sharp_{\nu}b )^{2}+ r_{0}^{2}(a-b)^{2}+
 \sum_{k=1}^{\infty}r_{k} \big[  a^{1-\frac{m_k}{2^k}}b^{\frac{m_k}{2^k}}
  - a^{1-\frac{m_k+1}{2^k}}b^{\frac{m_k+1}{2^k}}
   \big]^{2}.
\end{equation}

\end{theorem}
\begin{proof}
 By (\ref{y3}), we have
\begin{align*}
( a \nabla_{\nu} b)^{2} -  r_{0}^{2}(a-b)^{2} &=a^{2} \nabla_{\nu} b^{2}-r_{0} (a-b)^{2}\\
& \geq  ( a \sharp_{\nu} b)^{2} + \sum_{k=1}^{\infty}r_{k} \big[  a^{1-
\frac{m_k}{2^k}}b^{\frac{m_k}{2^k}}
  - a^{1-\frac{m_k+1}{2^k}}b^{\frac{m_k+1}{2^k}}
   \big]^{2}
\end{align*}
\end{proof}
\begin{theorem}\label{l4}
Let $ a, b $ be two positive real numbers and $ \nu \in (0, 1). $
Then
\begin{equation}\label{y6}
 ( a \nabla_{\nu} b)^{2} \leqslant ( a \sharp_{\nu}b )^{2}+ (1-r_{0})^{2}(a-b)^{2}-
 \sum_{k=1}^{\infty}r_{k} \big[  a^{1-\frac{m_k}{2^k}}b^{\frac{m_k}{2^k}}
  - a^{1-\frac{m_k+1}{2^k}}b^{\frac{m_k+1}{2^k}}
   \big]^{2}.
\end{equation}

\end{theorem}
\begin{proof}
  We have
\begin{align*}
( a \nabla_{\nu} b)^{2} & - (1- r_{0})^{2}(a-b)^{2} \\
&=a^{2} \nabla_{\nu} b^{2}-(1-r_{0}) (a-b)^{2}\\
& \leqslant  ( a \sharp_{\nu} b)^{2}+r_{0}(a-b)^{2}  \\
& \quad - \sum_{k=0}^{\infty}r_{k} \big[  a^{1-\frac{m_k}{2^k}}b^{\frac{m_k}{2^k}}
  - a^{1-\frac{m_k+1}{2^k}}b^{\frac{m_k+1}{2^k}}
   \big]^{2}\\
   & \ \ \ \ \  \ \ \ \ \ \ \ \ \ \ \text{ by  inequality (\ref{y4})}\\
   &=( a \sharp_{\nu} b)^{2}  -\sum_{k=1}^{\infty}r_{k} \big[  a^{1-
   \frac{m_k}{2^k}}b^{\frac{m_k}{2^k}}
     - a^{1-\frac{m_k+1}{2^k}}b^{\frac{m_k+1}{2^k}}
      \big]^{2}.\\
\end{align*}
\end{proof}
\section{Related operator inequalities}
Two state the operator versions of the inequalities obtained in
section 2, we need the following lemma.
\begin{lemma}\cite{f3}\label{l1}
Let $ X \in  \mathbb{B}(\mathcal{H}) $ be self-adjoint  and let $
f $ and $ g $ be continuous real functions such that $ f(t)
\geqslant g(t)$ for all $ t \in \sigma(X) $ (the spectrum of $ X
$). Then $ f(X) \geqslant g(X). $
\end{lemma}
Next, we give the  first result in this section, which is based on
Theorem \ref{l01} and is a refinement of Theorem 1 in \cite{jz}.
\begin{theorem}\label{t1}
Let $ A, B \in  \mathbb{B}(\mathcal{H}) $ be two positive
invertible operators and $ \nu \in (0, 1). $
\begin{equation}\label{1e}
 A \nabla_{\nu} B \geqslant   A \sharp_{\nu} B+
\sum_{k=0}^{\infty}r_{k} [ A\sharp_{\frac{m_k}{2^k}}B -2
A\sharp_{\frac{2m_k+1 }{2^{k+1}}}B +
A\sharp_{{\frac{m_k+1}{2^k}}}B].
\end{equation}

\end{theorem}
\begin{proof}
Choosing $ a=1, $ in Theorem \ref{l01},  we have
\begin{equation*}
 1-\nu+ \nu b \geqslant  b^{\nu}+
\sum_{k=0}^{\infty}r_{k}\big[\big(b^{\frac{m_k}{2^k}}
\big)^{\frac{1}{2}}-\big(b^{\frac{m_k+1}{2^k}}
\big)^{\frac{1}{2}}\big]^{2},
\end{equation*}
for any $ b > 0. $\\
If    $ X=A^{-\frac{1}{2}}BA^{-\frac{1}{2}}, $  then  $\sigma(X)
\subseteq (0, \infty). $ According to Lemma \ref{l1}, we get
\begin{equation*}
 (1-\nu)I+\nu X \geqslant  X^{\nu}+
\sum_{k=0}^{\infty}r_{k} [X^{\frac{m_k}{2^k}} -2X^{\frac{2m_k+1
}{2^{k+1}}} + X^{\frac{m_k+1}{2^k}}].
\end{equation*}
Multiplying both sides  by $ A^{\frac{1}{2}},$ we obtain
\begin{equation*}
 A \nabla_{\nu} B \geqslant   A \sharp_{\nu} B+
\sum_{k=0}^{\infty}r_{k} [ A\sharp_{\frac{m_k}{2^k}}B -2
A\sharp_{\frac{2m_k+1 }{2^{k+1}}}B +
A\sharp_{{\frac{m_k+1}{2^k}}}B].
\end{equation*}
This completes the proof.
\end{proof}

Since for all positive integer $ n, $
\small{
 $$ f(t) = \sum _{k=0}^{n} r_{k}\big[\big(t^{\frac{m_k}{2^k}}
\big)^{\frac{1}{2}}-\big(t^{\frac{m_k+1}{2^k}}
\big)^{\frac{1}{2}}\big]^{2}=\sum _{k=0}^{n} r_{k}  \big[
t^{\frac{m_k}{2^k}}
  -2t^{\frac{2m_k+1}{2^{k+1}}}
+   t^{\frac{m_k+1}{2^k}} \big]$$ }
 is a continuous function on $ [0, \infty) $ and $ A^{-\frac{1}{2}}B
 A^{-\frac{1}{2}} $ is a positive operator, then $ \sigma (f(A^{-\frac{1}{2}}
 BA^{-\frac{1}{2}}))  \subseteq [0, \infty). $  Thus
  $$ A^{\frac{1}{2}} f(A^{-\frac{1}{2}}BA^{-\frac{1}{2}}) A^{\frac{1}{2}}=
  \sum _{k=0}^{n} r_{k} (A\sharp_{\frac{m_k}{2^k}} B -2 A\sharp_{
  \frac{2m_k+1}{2^{k+1}}}B+  A\sharp_{\frac{m_k+1}{2^k}} B),$$
  is a  positive operator.  Then by inequality (\ref{1e}), we obtain
\begin{align*}
  A  \sharp_{\nu} B  \leqslant   A \sharp_{\nu} B +  \sum _{k=0}^{n} r_{k} (A\sharp_{
\frac{m_k}{2^k}} B -2 A\sharp_{\frac{2m_k+1}{2^{k+1}}}B+ A\sharp_{
\frac{m_k+1}{2^k}} B) \leqslant A \nabla_{\nu} B.
\end{align*}
and therefore
\begin{align}\label{a1}
  A  \sharp_{\nu} B  \leqslant   A \sharp_{\nu} B +  \sum _{k=0}^{\infty} r_{k} (A\sharp_{
\frac{m_k}{2^k}} B -2 A\sharp_{\frac{2m_k+1}{2^{k+1}}}B+ A\sharp_{
\frac{m_k+1}{2^k}} B)  \leqslant A \nabla_{\nu} B.
\end{align}
Replacing $ A $ and $ B $ by $ A^{-1} $ and $ B^{-1} $
respectively, we obtain \small{
\begin{align}\label{a2}
 &A^{-1}  \sharp_{\nu} B^{-1}  \nonumber\\
 &\leqslant   A^{-1} \sharp_{\nu} B^{-1}+  \sum _{k=0}^{\infty} r_{k} (A^{-1}\sharp_{\frac{m_k}{2^k}}B^{-1} -2 A^{-1}\sharp_{\frac{2m_k+1}{2^{k+1}}}
B^{-1}+  A^{-1}\sharp_{\frac{m_k+1}{2^k}} B^{-1}) \nonumber \\
& \leqslant A^{-1} \nabla_{\nu} B^{-1}.
\end{align}}

 Taking inverse in (\ref{a2}), we have \small{
\begin{align}\label{a3}
  & A  !_{\nu}B   \nonumber \\
& \leqslant   \{ A^{-1} \sharp_{\nu} B ^{-1}+  \sum _{k=0}^{n}
r_{k} (A^{-1}\sharp_{\frac{m_k}{2^k}} B^{-1} -2 A^{-1}\sharp_{
\frac{2m_k+1}{2^{k+1}}}B^{-1}+  A^{-1}\sharp_{\frac{m_k+1}{2^k}} B^{-1})\}^{-1}  \nonumber \\
& \leqslant A \sharp_{\nu} B.
\end{align}}
It is worth to mention that inequalities  (\ref{a1}),   (\ref{a2})   and  (\ref{a3})
are respectively refinements of inequalities (30)-(34) in \cite{jz}.\\
The following theorem is an operator version of Theorem \ref{l2}
and is a refinement of Theorem 2 in \cite{jz}.
\begin{theorem}
Let $ A, B \in  \mathbb{B}(\mathcal{H}) $ be two positive
invertible operators and $ \nu \in (0, 1). $
\begin{equation*}
 A \nabla_{\nu} B \leqslant   A \sharp_{\nu} B +(A-2A \sharp B + B  ) -
\sum_{k=0}^{\infty}r_{k} [ A\sharp_{\frac{m_k}{2^k}}B -2
A\sharp_{\frac{2m_k+1 }{2^{k+1}}}B +
A\sharp_{{\frac{m_k+1}{2^k}}}B].
\end{equation*}
\end{theorem}
\begin{proof}
By Theorem \ref{l2}, using the same ideas as in the proof of
Theorem \ref{t1}, we can get the result.
\end{proof}
\begin{corollary}
Let $ A, B \in  \mathbb{B}(\mathcal{H}) $ be two positive
invertible operators and $ \nu \in (0, 1). $ Then
\begin{equation*}
 A \nabla B \geqslant  H_{\nu}(A, B) +
\sum_{k=0}^{\infty}r_{k} [ H_{\frac{m_k}{2^k}}(A,B) -2
H_{\frac{2m_k+1 }{2^{k+1}}}(A,B) + H_{{\frac{m_k+1}{2^k}}}(A,B)].
\end{equation*}
and
\begin{align*}
 A \nabla B \leqslant  H_{\nu}(A, B) &+(A-2A \sharp B + B  ) \\
 &-
\sum_{k=0}^{\infty}r_{k} [ H_{\frac{m_k}{2^k}}(A,B) -2
H_{\frac{2m_k+1 }{2^{k+1}}}(A,B) + H_{{\frac{m_k+1}{2^k}}}(A,B)].
\end{align*}
\end{corollary}

\section{ Hilbert-Schmidt norm version}
In this  section, we obtain some inequalities for the Hilbert-Schmidt norm.
Applying Theorem \ref{l3}, we get the
following theorem that is a refinement of first inequalities in
(\ref{b1}) and (\ref{b2}).
\begin{theorem}\label{t2}
Let $ A, B, X \in  \mathbb{M}_{n}$  such that $ A $ and  $ B$  are two  positive
semidefinite matrices  and $ \nu \in (0, 1). $ Then
\begin{align*}
\| A^{1-\nu}XB^{\nu}\|_{2}^{2}+r_{0}^{2}\|AX-XB\|_{2}^{2}&+\sum _{k=1}^{\infty}r_{k}
\|A^{1-\frac{m_{k}}{2^{k}}}XB^{\frac{m_{k}}{2^{k}}}- A^{1-\frac{m_{k}+1}{2^{k}}}X
B^{\frac{m_{k}+1}{2^{k}}} \|_{2}^{2}\\
&\leqslant  \|(1-\nu)AX-\nu XB \|_{2}^{2}.
\end{align*}
\end{theorem}
\begin{proof}
Since $ A $ and  $ B $ are positive  semidefinite $ n\times n $ matrices,  there  exist  unitary matrices $ U, V  \in  \mathbb{M}_{n} $   such that
 $ A=U diag (\lambda_{1}, \ldots, \lambda_{n})U^{*} $ and    $ B=V diag (\mu_{1},
 \ldots, \mu_{n})V^{*}. $
Let $ Y=U^{*}XV=(y_{ij}). $ Then it's straightforward to check that
$$ (1-\nu)AX-\nu XB =U [((1-\nu)\lambda_{   i }+\nu \mu_{j})\circ Y] V^{*}, $$
 $$ AX-XB=U[(\lambda_{i}-\mu_{j})\circ Y] V^{*} $$
$$ A^{1-\nu} X B^{\nu}=U [(\lambda_{i}^{1- \nu}\mu_{j}^{\nu})\circ Y]V^{*}$$   and
$$A^{1-\frac{m_{k}}{2^{k}}}XB^{\frac{m_{k}}{2^{k}}}- A^{1-\frac{m_{k}+1}{2^{k}}}X
B^{\frac{m_{k}+1}{2^{k}}}=U[(\lambda_{i}^{1-\frac{m_{k}}{2^{k}}}
\mu_{j}^{\frac{m_{k}}{2^{k}}}- \lambda_{i}^{1-\frac{m_{k}}{2^{k}}}
\mu_{j}^{\frac{m_{k}}{2^{k}}}) \circ Y]V^{*}.$$ Utilizing the
unitarily invariant  property of $ \|.\|_{2} $ and Theorem
\ref{l3}, we have
\begin{align*}
 &\|(1-\nu)AX-\nu XB \|_{2}^{2}\\
 &=\|((1-\nu)\lambda_{  i }+\nu \mu_{j})\circ Y \|_{2}^{2}\\
 &=\sum_{i,j=1}^{n} ((1-\nu)\lambda_{ i }+\nu \mu_{j})^2|y_{ij}|^{2}\\
 & \geqslant\sum_{i,j=1}^{n} \left \{(\lambda_{i}^{1- \nu}\mu_{j}^{\nu})^2 + r_{0}^{2}
 (\lambda_{i}-\mu_{j})^{2}+ \sum _{k=1}^{\infty}r_{k}(\lambda_{i}^{1-\frac{m_{k}}{2^{k}}}
 \mu_{j}^{\frac{m_{k}}{2^{k}}}-
 \lambda_{i}^{1-\frac{m_{k}}{2^{k}}} \mu_{j}^{\frac{m_{k}}{2^{k}}})^{2}    \right\}|y_{ij}|^{2}\\
 &=\sum_{i,j=1}^{n}(\lambda_{i}^{1- \nu}\mu_{j}^{\nu})^2|y_{ij}|^{2}+ \sum_{i,j=1}^{n}  r_{0}^{2}
 (\lambda_{i}-\mu_{j})^{2}|y_{ij}|^{2}\\
 & \quad +   \sum_{i,j=1}^{n} \left   \{   \sum _{k=1}^{\infty}r_{k}(\lambda_{i}^{1-
 \frac{m_{k}}{2^{k}}} \mu_{j}^{\frac{m_{k}}{2^{k}}}-
  \lambda_{i}^{1-\frac{m_{k}}{2^{k}}} \mu_{j}^{\frac{m_{k}}{2^{k}}})^{2}  |y_{ij}|^{2}    \right\}\\
  &=\sum_{i,j=1}^{n}(\lambda_{i}^{1- \nu}\mu_{j}^{\nu})^2|y_{ij}|^{2}+ \sum_{i,j=1}^{n}
  r_{0}^{2}(\lambda_{i}-\mu_{j})^{2}|y_{ij}|^{2}\\
  & \quad +   \sum_{k=1}^{\infty} \left   \{   \sum _{i, j=1}^{n}r_{k}
  (\lambda_{i}^{1-\frac{m_{k}}{2^{k}}} \mu_{j}^{\frac{m_{k}}{2^{k}}}-
   \lambda_{i}^{1-\frac{m_{k}}{2^{k}}} \mu_{j}^{\frac{m_{k}}{2^{k}}})^{2}  |y_{ij}|^{2}    \right\}\\
   &= \| A^{1-\nu}XB^{\nu}\|_{2}^{2}+r_{0}^{2}\|AX-XB\|_{2}^{2} \\
   & \quad +\sum _{k=1}^{\infty}r_{k} \|A^{1-\frac{m_{k}}{2^{k}}}XB^{\frac{m_{k}}{2^{k}}}
   - A^{1-\frac{m_{k}+1}{2^{k}}}XB^{\frac{m_{k}+1}{2^{k}}} \|_{2}^{2}.
\end{align*}
So,  the proof is  complete.
\end{proof}

The last theorem is a refinement of second inequalities in
(\ref{b1})  and (\ref{b2}).
\begin{theorem}
Let $ A, B, X \in  \mathbb{M}_{n}$  such that $ A $ and  $ B$  are two  positive
semidefinite matrices  and $ \nu \in (0, 1). $ Then
\begin{align*}
 \|(1-\nu)AX-\nu XB \|_{2}^{2} & \leqslant  \| A^{1-\nu}XB^{\nu}\|_{2}^{2}+(1- r_{0})^{2}
 \|AX-XB\|_{2}^{2}\\
& \quad -\sum _{k=1}^{\infty}r_{k} \|A^{1-\frac{m_{k}}{2^{k}}}XB^{\frac{m_{k}}{2^{k}}}-
 A^{1-\frac{m_{k}+1}{2^{k}}}XB^{\frac{m_{k}+1}{2^{k}}} \|_{2}^{2}.
\end{align*}
\end{theorem}
\begin{proof}
By Theorem \ref{l4} and  using  the same idea as in the proof of
Theorem \ref{t2},  we can obtain the desired  result.
\end{proof}
\bibliographystyle{amsplain}

\end{document}